\documentclass[a4paper,10pt]{amsart}
\usepackage{amsmath,amsthm,amssymb}
\usepackage[latin1]{inputenc}
\usepackage[T1]{fontenc}
\usepackage[all]{xy}
\usepackage{mathrsfs}

\newtheorem{thm}{Theorem}[section]
\newtheorem{cor}{Corollary}[section]

\newtheorem{prop}{Proposition}[section]

\theoremstyle{definition}
\newtheorem{rem}{Remark}[section]

\newcommand{\triv}{{\mathbf{1}}}
\def\R{\mathbb R}

\def\Z{\mathbb Z}
\def\A{\mathbb A}
\def\Q{\mathbb Q}
\def\C{\mathbb C}

\def\wt{\widetilde}
\def\ira{\stackrel{\sim}{\rightarrow}}
\def\hra{\hookrightarrow}

\def\ra{\rightarrow}

\def\H{\mathbb H}

\def\g{\mathfrak g}
\def\t{\mathfrak t}

\def\n{\mathfrak n}
\def\a{\mathfrak a}
\def\b{\mathfrak b}
\def\q{\mathfrak q}
\def\om{\omega}

\def\l{\mathfrak l}
\def\m{\mathfrak m}
\def\k{\mathfrak k}
\def\p{\mathfrak p}

\def\G{G_{s}}

\def\<{\langle}
\def\>{\rangle}

\title[]%
{Automorphic Forms, Cohomology and CAP Representations. The Case $GL_2$ over a definite quaternion
algebra}
\author{Harald Grobner}
\date{\today}
\curraddr{Max-Planck Institut f\"ur Mathematik\\ Vivatsgasse 7\\ 53111 Bonn, Germany}
\email{harald.grobner@univie.ac.at}

\thanks{The author was supported by the Austrian Science Found (FWF), project no. P21090}

\begin{document}

\begin{abstract}
In this paper we fully describe the cuspidal and the Eisenstein
cohomology of the group $G=GL_2$ over a definite quaternion algebra
$D/\Q$. Functoriality is used to show the existence of
residual and cuspidal automorphic forms, having cohomology in degree
$1$. The latter ones turn out to be CAP-representations, though $G$
satisfies Strong Multiplicity One. A non-vanishing result on
intertwining operators of induced representations will serve as a
starting point for further investigations concerning rationality of
critical $L$-values.
\end{abstract}
\maketitle

\section*{Introduction}
This paper, in a sense, follows up on a work of G. Harder,
\cite{hardergl2}. He considered the so-called Eisenstein cohomology
of arithmetic subgroups of $GL_2$ defined over a number field
$\mathbb K$. Here we want to describe the cohomology of arithmetic
congruence subgroups $\Gamma$ of $G=GL_2'$, the group of invertible
$2\times 2$-matrices with entries in a definite quaternion algebra
$D/\Q$. Furthermore, we do not restrict ourselves to determining the
Eisenstein cohomology of such groups $\Gamma$, but also describe
their cuspidal cohomology. \\\\ The Eisenstein cohomology of
$G=GL_2'$ with respect to a finite-dimensional, irreducible
complex-rational representation $E$ of $G(\R)$, denoted
$H^*_{Eis}(G,E)$, is spanned by classes represented by (a bunch of)
Eisenstein series, residues or derivatives of such. As suggested by
the last sentence, in order to determine the very nature of a
representative of an Eisenstein class, one needs to decide whether a
given Eisenstein series is or is not holomorphic at a certain point.
As shown by Langlands, this means to describe the residual spectrum
of $G$. \\\\ Here the first problem arises. For a {\it quasi-split}
group the so-called Langlands-Shahidi method is available, which
expresses the poles of Eisenstein series by poles of a certain
intertwining operator $M(s,\wt\pi)$. This operator links two
representations induced from a cuspidal automorphic
representation $\wt\pi$ of a proper Levi subgroup. If $\wt\pi$ is {\it generic}, then the Langlands-Shahidi method provides a normalization of $M(s,\wt\pi)$ defined using $L$-functions, which theoreticlally allows one to read off the poles of $M(s,\wt\pi)$. Clearly,
$G=GL_2'$ is not quasi-split and - furthermore - there are even many
cuspidal automorphic representations of $L=D^*\times D^*$ (sitting
inside $G$ as the Levi subgroup of the unique proper standard
parabolic $\Q$-subgroup $P$), which have non-generic local
components. Still, using \cite{ich2}, Prop. 3.1, we are able to
reduce the task of normalizing the global intertwining operator
$$M(s,\wt\pi):\textrm{Ind}^{G(\A)}_{P(\A)}[\sigma\otimes\tau\otimes e^{sH_P(.)}]\ra\textrm{Ind}^{G(\A)}_{P(\A)}
[\tau\otimes\sigma\otimes e^{-sH_P(.)}]$$ to normalizing its local
components at split places. As before, $\wt\pi=\sigma\otimes\tau$ is
a cuspidal automorphic representation of $L=D^*\times D^*$. To
overcome the second problem, namely to actually normalize it at
non-generic split places, we use the recent work of N. Grbac,
\cite{nevenGn}. With this achievement at hand, we are able to
classify the cuspidal automorphic representations $\wt\pi$ of $L$,
which give rise to residual Eisenstein series and show where their
poles are located, see Proposition \ref{prop:poles}.\\ This gives
the first main theorem of this paper, see Theorem \ref{theoremeis},
which comprehensively describes to Eisenstein cohomology of $GL_2'$
with respect to arbitrary coefficients $E$. In particular, necessary
and sufficient conditions for the existence of non-trivial residual
Eisenstein cohomology classes are established and it is explained
how they are distributed (resp. separated from holomorphic
Eisenstein classes) in (resp. by) the degrees of cohomology.
\\\\Degree one seems to be of most interest, since Eisenstein
series contributing to it must form a non-trivial residual
automorphic representation of $G$. After having shown their
existence by having proved the non-vanishing of $H^1_{Eis}(G,E)$ in
Theorem \ref{theoremeis}, we explicitly construct such residual
representations in Theorem \ref{thm:q=1}. Here we use functoriality,
more precisely, the global Jacquet-Langlands Correspondence from
$GL_2'$ to $GL_4$, as it was only recently developed by I. Badulescu
and D. Renard, \cite{ioan}. \\\\As a next step we shortly reconsider Eisenstein cohomology in degree $q=2$ and prove a theorem (cf., Theorem
\ref{t:nonzero}) that serves as the starting point for investigating
rationality results of critical values of $L$-functions. It also answers in our particular case of $G=GL_2'$ a question raised by G. Harder.
\\\\Finally, in section 5, we describe the cuspidal cohomology of
$G$, denoted $H^q_{cusp}(G,E)$. Its elements, i.e., cuspidal
cohomology classes, are represented by cuspidal automorphic forms.
As it follows from our classification of the cohomological unitary
dual of $G(\R)$, see Proposition \ref{theoremcohunit}, cuspidal
cohomology classes gather in pairs $(q,5-q)$, $q=1,2$, symmetrically
around the middle degree of cohomology. They are quite different in
their nature: The classes in degrees $(2,3)$ are represented by a
representation with tempered archimedean component, while the
classes in degree $(1,4)$ will give rise to the existence of
cuspidal automorphic representations with {\it non}-tempered
archimedean component.\\\\Again we use functoriality, {\it expressis
verbis}, the global Jacquet-Langlands Correspondence, to show the
existence of cuspidal automorphic representations giving rise to
cohomology in degrees $q=1$ and $4$, see Theorem
\ref{thm:nontempcusp}. Those serve as examples for cohomological
CAP-representations, although $G$ satisfies Strong Multiplicity One.
Further, it proves the non-vanishing of $H^q_{cusp}(G,E)$,
$q=1,4$.\\ In the last part of this paper, we prove the
non-vanishing of cuspidal cohomology in degrees $q=2$ and $3$. It
follows from the twisted version of Arthur's Trace Formula,
respectively applying the results of \cite{borlabschw} or
\cite{barsp}.
\newpage

\section{Basic group data}
\subsection{The groups}\label{sec:1}
Let $D$ be a quaternion algebra over the field of rational numbers
$\Q$ and $S(D)$ be the set of places over which $D$ does not split.
By the Brauer-Hasse-Noether Theorem, $S(D)$ determines $D$ up to
isomorphism and has always finite and even cardinality (cf.
\cite{platrap} Thm. 1.12). We assume throughout this paper that the
archimedean place $\infty$ is in $S(D)$, so $D\otimes_\Q\R\cong\H$,
the usual Hamilton quaternions and therefore $D$ is not split over
$\Q$. If $x\in D$, we set $\nu(x):=x\overline x$, where
$x\mapsto\overline x$ is the conjugation within $D$. The determinant
$\det'$ of an $n\times n$-matrix $X\in M_n(D)$, $n\geq 1$, is then
the generalization of $\nu$ to matrices:
$\det'(X):=\det(\varphi(X\otimes 1))$, for some isomorphism
$\varphi: M_n(D)\otimes_\Q\overline\Q\ira M_{2n}(\overline\Q)$. It
is independent of $\varphi$ (see, e.g., \cite{platrap} 1.4.1) and a
rational polynomial in the coordinates of the entries of $X$. So the
group

$$G(\Q)=\{X\in M_n(D)|\textrm{det}'(X)\neq 0\}$$
defines an algebraic group over $\Q$, which we denote $GL_n'$. It is
reductive and has a natural simple subgroup $SL_n'$, the group of
matrices of determinant $1$. We will consider $G=GL_2'$ and $\G=SL_2'$.\\\\
Fix a maximal $\Q$-split tours $T\subset G$. It is of the form
$T\cong GL_1\times GL_1$ embedded diagonally into $G$. As the centre
$Z$ of $G$ consists of diagonal scalar matrices, i.e., $Z\cong GL_1$
embedded diagonally into $G$,
$$A:=T\cap SL_2'=\left\{\left( \begin{array}{cc}
a & 0\\
0 & a^{-1}
\end{array}\right)\right\}$$
is a maximal $\Q$-split torus of $SL_2'$ and $T=Z\cdot A$. Observe
that $T$ is also maximally split over $\R$. Therefore we may
identify the set of $\Q$-roots $\Delta_\Q$ and the set of $\R$-roots
$\Delta_\R$ of $G$. They are given as
$\Delta_\Q=\Delta_\R=\{\pm(\alpha-\beta)\}$, where $\alpha$ (resp.
$\beta$) extracts the first (resp. the second) diagonal entry of
$T$. On the level of the simple group $SL_2'$, the two roots
degenerate to $\pm 2\alpha$. We chose positivity in the usual way as
$\Delta^+_\Q=\Delta^+_\R=\{\alpha-\beta\}$. It follows that
$rk_\Q(SL_2')=1$. Furthermore, the $\Q$-Weyl group $W_\Q$ of $G_s$
(and $G$) consists precisely of $1$ and one non-trivial quadratic
element $w$ which acts by exchanging $\alpha$ and $\beta$:
$W_\Q=\{1,w\}$.
\\\\ Observe that $T$ and $A$ are not maximally
split over $\C$. In fact, $GL_2'(\C)\cong
GL_4(\C)$ and so we need to multiply $T$ (resp. $A$) by a
two-dimensional torus $B\cong GL_1\times GL_1$, in order to get a
maximally $\C$-split torus of $GL_2'$ (resp. $SL_2'$). Let
$H:=A\cdot B\hra SL_2'$. Then $H(\C)$ can be viewed as
$$H(\C)=\left\{\left( \begin{array}{cc}
(a,b_1) & 0\\
0 & (a^{-1},b_2)
\end{array}\right)\right\},$$
where the pair $(a,b_i)\in \C^*\times\C^*$, $i=1,2$, represents the first two coordinates of
the quaternion division algebra $D$. The group $H(\C)$ is a Cartan subgroup of $G_s(\C)\cong SL_4(\C)$ (and $G(\C)$).
If we set for $h\in H(\C)$,
$$\varepsilon_i(h)=\left\{
\begin{array}{ll}
b_1+a & i=1\\
a-b_1 & i=2\\
b_2-a & i=3\\
-b_2-a & i=4\\
\end{array}
\right.$$ then the set of absolute (i.e. $\C$-)roots of $G_s=SL_2'$
is generated by the following set of simple roots
$\Delta^\circ=\{\alpha_1:=\varepsilon_1-\varepsilon_2,\alpha_2:=\varepsilon_2-\varepsilon_3,
\alpha_3:=\varepsilon_3-\varepsilon_4\}$, where we observe that the
choice of positivity given by $\Delta^\circ$ is compatible with the
one on $\Delta_\Q$. Accordingly, $rk_{\C}(SL_2')=3$ and the
fundamental weights of $G_s$ are
$\omega_1=\frac{3}{4}\alpha_1+\frac{1}{2}\alpha_2+\frac{1}{4}\alpha_3$,
$\omega_2=\frac{1}{2}\alpha_1+\alpha_2+\frac{1}{2}\alpha_3$ and
$\omega_3=\frac{1}{4}\alpha_1+\frac{1}{2}\alpha_2+\frac{3}{4}\alpha_3$.
As usual their sum is denoted by $\rho$. For the simple reflection
corresponding to the root $\alpha_i$ we write $w_i$ and the Weyl
group generated by these reflections is called $W=W(G,H)=W(G_s,H)$.

\subsection{Real Groups and Symmetric Spaces}\label{sec:realgrps}
The real Lie groups $G(\R)$ and $G_s(\R)$ are isomorphic to
$GL_2(\H)$ and $SL_2(\H)$, respectively. If we fix a Cartan
involution $\theta(x):={}^t\overline x^{-1}$ on $G$, then the
corresponding maximal compact subgroups $K$ of $G(\R)$ and $K_s$ of
$G_s(\R)$ are isomorphic to $K\cong K_s\cong Sp(2)$. As
$G(\R)=Z(\R)^\circ\times G_s(\R)$, the corresponding symmetric
spaces $X=Z(\R)^\circ\backslash G(\R)/K\cong \R_{>0}^*\backslash
GL_2(\H)/Sp(2)$ and $X_s=G_s(\R)/K_s=SL_2(\H)/Sp(2)$ are
diffeomorphic, whence we will not distinguish between them. Their
common real dimension is $\dim_\R X=5$.

\subsection{}
Throughout this paper $E=E_{\lambda}$ denotes an irreducible,
finite-dimensional representation of $G_s(\R)$ on a complex vector
space determined by its highest weight
$\lambda=\sum_{i=1}^3c_i\alpha_i$. We can and will view $E$ also as
a representation of $G(\R)$ by extending it trivially on
$Z(\R)^\circ$. Furthermore, we will always assume that $E$ is the
complexification of an algebraic representation of $G_s/\Q$.

\subsection{The Parabolic Subgroup}\label{sec:parabolic}
Since $G_s=SL_2'$ is a $\Q$-rank one group, we only have one single
proper standard parabolic $\Q$-subgroup $P$ of $G$ (resp.
$P_s=P\cap\G$ of $\G$). It is given by

$$P=\left\{\left( \begin{array}{cc}
a & n\\
0 & b
\end{array}\right)| a,b\in GL_1', n\in D\right\}$$
with Levi- (resp. Langlands-) decomposition $P=LN$ (resp. $P=MTN$), where

$$L=GL_1'\times GL_1'$$
$$M=M_s=SL_1'\times SL_1'$$
and
$$N=N_s\cong D \textrm{ (as additive group)}.$$
The set of absolute simple roots of $M=M_s$ with respect to $B=H\cap
M$ is $\Delta_M^\circ=\{\alpha_1,\alpha_3\}$. For later use we already define and determine
\begin{eqnarray*}
W^P:=W^{P_s}&:= & \{w\in W| w^{-1}(\alpha_i)>0\quad i=1,3\}\\
            & = & \{id, w_2, w_2w_1,w_2w_3,w_2w_1w_3, w_2w_1w_3w_2\}
\end{eqnarray*}
This set is called the set of Kostant representatives, cf. \cite{bowa} III 1.4.

\section{Generalities on the automorphic cohomology of $G$}

\subsection{A space of automorphic forms}\label{sec:firstdec}
Having fixed some basic facts and notation concerning $G$ and $G_s$,
we will now delve into cohomology, the object to be studied in this
paper. The main reference for this general section is \cite{schwfr}.
For the special case of a group of rank one, the reader may also
consult the first sections of the author's
paper \cite{ich2}, where the results of \cite{schwfr} are sumarized for these groups.\\\\
Let $\mathcal U(\g)$ be the universal enveloping algebra of the complexification $\g_\C=\g\otimes_\R\C$ of
the real Lie algebra $\g=\mathfrak{gl}_2(\H)$ of $G(\R)$. Let
$\mathcal Z(\g)$ be the center of  $\mathcal U(\g)$ and $\mathcal
J\lhd\mathcal Z(\g)$ an ideal of finite codimension in $\mathcal
Z(\g)$. We denote by $\mathcal A(G)=\mathcal A(G,\mathcal J)$ the
space of those adelic automorphic forms on $Z(\R)^\circ
G(\Q)\backslash G(\A)$ (in the sense of \cite{bojac}, 4) which are
annihilated by some power of $\mathcal J$. As we will be only
interested in automorphic forms which have non-trivial
$(\g_s,K)$-cohomology with respect to $E$, we take $\mathcal
J\lhd\mathcal Z(\g)$ to be the ideal annihilating the dual
representation
$\check{E}$, cf. \cite{schwfr} Rem. 3.4. \\\\
Having fixed our choice for $\mathcal J$, the space $\mathcal A(G)$
has a certain decomposition along the two $G(\Q)$-conjugacy classes of parabolic $\Q$-subgroups of $G$ (represented by $P$ and $G$), which we shall now describe. First of all we recall that the group
$G(\A_f)$ of finite adelic points acts on the space $C(G(\A_f))$ of
continuous functions $f:G(\A_f)\ra\C$ by right translation.
Topologised via the inductive limit
$$C^\infty(G(\A_f))=\varinjlim C(G(\A_f))^{K_f},$$
$K_f$ running over
all open, compact subgroups of $G(\A_f)$, the space
$C^\infty(G(\A_f))$ of smooth functions $f:G(\A_f)\ra\C$ is a
complete Hausdorff, locally convex vector space with semi-norms
$|.|_\alpha$ say, but generally not Fr\'echet. As a consequence,
$$C^\infty(G(\R),C^\infty(G(\A_f)))=C^\infty(G(\A))\quad \textrm{and}\quad C^\infty(Z(\R)^\circ G(\Q)\backslash G(\A))$$
are Hausdorff, locally convex
vector spaces carrying smooth $G(\R)$- and $G(\A_f)$-actions via
right translation. We recall that a left $Z(\R)^\circ G(\Q)$-invariant smooth
function $f\in C^\infty(G(\A))= C^\infty(G(\R),C^\infty(G(\A_f)))$ is of {\it
uniform moderate growth}, if $f$ satisfies
$$\forall D\in \mathcal U(\g), \forall |.|_\alpha \quad\exists N=N(f,|.|_\alpha), C=C(f,D,|.|_\alpha)\in\R_{\geq 0}:\quad |Df(g)|_\alpha< C\Vert g\Vert^N $$
for all $g\in G(\R).$ Here $\Vert g\Vert=\sqrt{\det'(g)^{-2}+2\cdot
tr(g\cdot {}^t\overline g)}$ is the usual norm on $G(\R)$, cf.
\cite{bojac} 1.2. Now, let $V_G$ be the space of smooth functions
$f\in C^\infty(Z(\R)^\circ G(\Q)\backslash G(\A))$ which are of
uniform moderate growth. It can be decomposed as
$$V_G=V_G(G)\oplus V_G(P),$$
where for $Q\in\{P,G\}$, $V_G(Q)$ denotes the space of $f\in V_G$
which are negligible along every parabolic $\Q$-subgroup $P'$ of $G$
not conjugate to $Q$. (This latter condition means that the constant
term of $f$ along $P'$ is orthogonal to the space of cusp forms of
the Levi subgroup of $P'$, cf. \cite{schwfr} 1.1.) A proof of this
result in a much more general context was already given by R. P.
Langlands in a letter to A. Borel, but may also be found in
\cite{borlabschw}, Thm. 2.4. Declaring
$$\mathcal A_{cusp}(G):=V_G(G)\cap\mathcal A(G)\quad \textrm{and}\quad\mathcal A_{Eis}(G):=V_G(P)\cap\mathcal A(G), $$
we therefore get the desired decomposition of $\mathcal A(G)$ as $(\g_s,K,G(\A_f))$-module:

$$\mathcal A(G)=\mathcal A_{cusp}(G)\oplus\mathcal A_{Eis}(G).$$
We remark that $\mathcal A_{cusp}(G)$ consists precisely of all
cuspidal automorphic forms in $\mathcal A(G)$ (which explains the
subscript), so finding the above decomposition amounted essentially
in finding an appropriate complement - namely $\mathcal A_{Eis}(G)$
- to the space of cuspidal automorphic forms in $\mathcal A(G)$. We
will explain the choice of the subscript of $\mathcal A_{Eis}(G)$ in
the next section.

\subsection{Eisenstein series}\label{sec:Eisseries}
The space $\mathcal A_{Eis}(G)=V_G(P)\cap\mathcal A(G)$
has a further description using Eisenstein series
attached to cuspidal automorphic representations of $L(\A)$. Therefore we need some technical assumptions:\\\\
Let $Q$ be a proper parabolic $\Q$-subgroup of $G$ with Levi- (resp.
Langlands-) decomposition $Q=L_QN_Q=M_QT_QN_Q$. In analogy to our
section \ref{sec:parabolic} we decompose the central torus $T_Q$ of
$L_Q$ as $T_Q=A_QZ$ with $A_Q=G_s\cap T_Q$. The torus $A_Q$ acts on
$N_Q$ by the adjoint representation and hence defines a set of
weights denoted by $\Delta(Q,A_Q)$. Summing up the halfs of these
weights gives a character $\rho_Q$. \\\\Now, let $\varphi_{Q}$ be a
finite set of irreducible representations $\pi=\chi\widetilde\pi$ of
$L_Q(\A)$, given by the following data:
\begin{enumerate}
\item $\chi:A_Q(\R)^\circ\ra\C^*$ is a continuous character. Since $T_Q(\R)^\circ=A_Q(\R)^\circ\times Z(\R)^\circ$,
      we can and will view the derivative $d\chi$ also as a character of
      $\t_Q$.
\item $\widetilde\pi$ is an irreducible, unitary $L_Q(\A)$-subrepresentation
of\\
      $L^2_{cusp}(L_Q(\Q)T_Q(\R)^\circ\backslash L_Q(\A))$, the cuspidal spectrum
      of $L_Q(\A)$. We assume furthermore, that the central character of $\wt\pi$
      induces a continuous, unitary homomorphism $T_Q(\Q)T_Q(\R)^\circ\backslash T_Q(\A)\ra \C^*$
      and that the infinitesimal character of $\wt\pi$ matches the one of the dual of
      an irreducible $M(\R)$-subrepresentation of the Lie algebra cohomology $H^*(\mathfrak{n}_Q,E)$.
\end{enumerate}
In short, these above conditions only mean that $\widetilde\pi$ is a
unitary, cuspidal automorphic representation of $L_Q(\A)$ whose
central and infinitesimal character satisfy the above
obstructions.\\\\ Finally, three further ``compatibility
conditions'' have to be satisfied between these sets $\varphi_{Q}$,
skipped here and listed in \cite{schwfr}, 1.2. The family of all
collections $\varphi=\{\varphi_{Q}\}$ of such finite sets is denoted
$\Psi$.\\\\ Let ${\rm I}_{Q,\wt\pi}:={\rm
Ind}^{G(\A)}_{Q(\A)}[\wt\pi]$ (normalized induction). For a
$K$-finite function $f\in {\rm I}_{Q,\wt\pi}$,
$\Lambda\in(\t_Q)_\C^*$ and $g\in G(\A)$ an Eisenstein series is (at
least) formally defined as

$$E(f,\Lambda)(g)=\sum_{\gamma\in Q(\Q)\backslash G(\Q)}f(\gamma g) e^{\< \Lambda, H_Q(\gamma g)\>}.$$
Here, $H_Q: G(\A)\ra (\t_Q)_\C^*$ is the usual Harish-Chandra height
function. If we set
$\t_Q^+:=\{\Lambda\in(\t_Q)_\C^*|Re(\Lambda)\in\rho_Q+C\}$, where
$C$ equals the open, positive Weyl-chamber with respect to
$\Delta(Q,A_Q)$, the series converges absolutely and uniformly on
compact subsets of $G(\A)\times (\t_Q^*)^+$. It is known that
$E(f,\Lambda)$ is an automorphic form there and that the map
$\Lambda\mapsto E(f,\Lambda)(.)$ can be analytically continued to a
meromorphic function on all of $(\t_Q)_\C^*$, cf. \cite{moewal} p.
140 or \cite{langbook} \S 7. It has only a finite number of at most
simple poles for $\Lambda\in (\a_Q)_\C^*\hookrightarrow(\t_Q)_\C^*$
with $Re(\Lambda)\geq 0$, see \cite{moewal} Prop. IV.1.11. If there
is a singularity at such a $\Lambda=\Lambda_0\in
(\a_Q)_\C^*\cong\C$, the residue ${\rm
Res}_{\Lambda_0}E(f,\Lambda)(g)$ is a smooth function and
square-integrable modulo $Z(\R)^\circ$, cf. \cite{moewal}
I.4.11. It defines hence an element of $L^2_{res}(G(\Q)Z(\R)^\circ\backslash G(\A))$, the residual spectrum of $G(\A)$.\\\\
Now we are able to give the desired description of $\mathcal A_{Eis}(G)$ by Eisenstein series: For
$\pi=\chi\widetilde\pi\in\varphi_P$ let $\mathcal A_\varphi(G)$
be the space of functions, spanned by all possible residues and
derivatives of Eisenstein series defined via all $K$-finite $f\in
{\rm I}_{P,\wt\pi}$ at the value $d\chi$ with $Re(d\chi)\geq 0$. It is a
$(\g_s,K,G(\A_f))$-module. Thanks to the functional equations (see
\cite{moewal} IV.1.10) satisfied by the Eisenstein series
considered, this is well defined, i.e., independent of the choice of
a representative for the class of $P$ (whence we took $P$ itself)
and the choice of a representation $\pi\in\varphi_P$. Finally, we
get

\begin{thm}[\cite{schwfr}, Thm. 1.4; \cite{moewal} III, Thm. 2.6]\label{thm:decEiscoh}
There is a direct sum decomposition as $(\g_s,K,G(\A_f))$-module
$$\mathcal A_{Eis}(G)=\bigoplus_{\varphi\in\Psi}\mathcal A_{\varphi}(G)$$
\end{thm}

\subsection{Definition of automorphic cohomology}\label{sec:defcoh}
Recall from section \ref{sec:firstdec} that $\mathcal A(G)$ is a
$(\g_s,K,G(\A_f))$-module. We can define:\\
The $G(\A_f)$-module
$$H^q(G,E):=H^q(\g_s,K,\mathcal A(G)\otimes E)$$
is called the {\it automorphic cohomology} of $G$ with respect to
$E$. Its natural complementary $G(\A_f)$-submodules
$$H^q_{cusp}(G,E):=H^q(\g_s,K,\mathcal A_{cusp}(G)\otimes E)\quad {\rm and }\quad H^q_{Eis}(G,E):=H^q(\g_s,K,\mathcal A_{Eis}(G)\otimes E)$$
are called {\it cuspidal cohomology} and {\it Eisenstein
cohomology}, respectively. We recall that the space $\mathcal
A_{cusp}(G)$ consists of smooth and $K$-finite functions in\\
$L_{cusp}^2(G(\Q)Z(\R)^\circ\backslash G(\A))$, which itself, acted
upon by $G(\A)$ via right translation, decomposes into a direct
Hilbert sum of irreducible admissible $G(\A)$-representations $\pi$,
each of which occruing with finite multiplicity $m(\pi)$. By Thm.
18.1.(b) of \cite{ioan} we have (Strong) Multiplicity One for
automorphic representations of $G(\A)$ appearing in the discrete
spectrum, whence we get by \cite{bowa}, XIII, a direct sum decomposition

\begin{equation}\label{eq:cuspcohdec}
H^q_{cusp}(G,E)=\bigoplus_{\pi}H^q(\g_s,K,\pi\otimes E),
\end{equation}
the sum ranging over all (equivalence classes of) cuspidal automorphic representations $\pi$ of $G(\A)$.\\\\
Using Theorem \ref{thm:decEiscoh} we can also decompose Eisenstein cohomology.

\begin{equation}\label{eq:Eiscohdec}
H^q_{Eis}(G,E)=\bigoplus_{\varphi\in\Psi}H^q(\g_s,K,\mathcal A_{\varphi}(G)\otimes E).
\end{equation}

\begin{rem}
The reason underlying these two decomposition is essentially the
same. In fact, one can think of (\ref{eq:cuspcohdec}) as a
degenerate version of (\ref{eq:Eiscohdec}), by rereading section
\ref{sec:Eisseries} and replacing in mind each proper parabolic
$\Q$-subgroup by $G$ itself. Chronologically, however, the approach
was as presented here.
\end{rem}

\section{Eisenstein cohomology}
By \eqref{eq:cuspcohdec} and \eqref{eq:Eiscohdec} it is clear, how
one can (at least in principle) describe the cuspidal cohomology and
the Eisenstein cohomology of $G$. One has to determine the
individual $G(\A_f)$-submodules $H^q(\g_s,K,\pi\otimes E)$ and
$H^q(\g_s,K,\mathcal A_{\varphi}(G)\otimes E)$ (notation as in
section \ref{sec:defcoh}). In this section we want to review a
method how to construct Eisenstein cohomology, using the notion of
``$(\pi,w)$-types'' and carry it out concretely. The next subsection
is devoted to the definition of the latter.

\subsection{Classes of type $(\pi,w)$}
Take $\pi=\chi\widetilde\pi\in\varphi_P$ and consider the symmetric
tensor algebra
$$S_\chi(\t^*)=\bigoplus_{n\geq 0} \bigodot^n\t_\C^*,$$ $\bigodot^n\t_\C^*$
being the symmetric tensor product of $n$ copies of $\t^*_\C$, as
module under $\t_\C$: Via the natural identification
$\t_\C\ira\t_\C^*$ given by the standard bracket $\<.,.\>$ it is a
$\t_\C$-module acted upon by $\xi\in\t_\C\cong\t_\C^*$ via
multiplication with $\<\xi,d\chi\>+\xi$ (within the symmetric tensor
algebra). This explains the subscript ``$\chi$''. We extend this
action trivially on $\l_\C$ and $\n_\C$ to get an action of the Lie
algebra $\p_\C$ on the Banach space $S_\chi(\t^*)$. We may also define a $ P(\A_f)$-module structure
via the rule
$$q\cdot X= e^{\<d\chi,H_P(q)\>}X,$$
for $q\in P(\A_f)$ and $X\in S_\chi(\t^*)$. There is
a continuous linear isomorphism
$${\rm Ind}_{P(\A)}^{G(\A)}[\wt\pi\otimes S_\chi(\t^*)]\ira {\rm
I}_{P,\wt\pi}\otimes S_\chi(\t^*),$$ so in particular one can view
the right hand side as a $G(\A)$-module by transport of structure.
Doing this, it is shown in \cite{franke}, pp. 256-257, that

\begin{equation}\label{dchi}
H^q(\g_s,K,{\rm I}_{P,\wt\pi}\otimes S_\chi(\t^*)\otimes E)\cong
$$$$ \bigoplus_{\substack{w\in W^P\\ -w(\lambda+\rho)|_{\a_\C}=d\chi}}
\textrm{Ind}_{P(\A_f)}^{G(\A_f)}\left[H^{q-l(w)}(\m,K_M,\widetilde\pi_\infty\otimes F_w)\otimes\C_{d\chi}\otimes
\widetilde\pi_f\right].
\end{equation}
Some notation needs to be explained: Above, $K_M=K\cap M(\R)$,
$F_w$ is the finite dimensional representation of
$M(\C)$ with highest weight $\mu_w:=w(\lambda+\rho)-\rho|_{\b_\C}$ and
$\C_{d\chi}$ the one-dimensional, complex $P(\A_f)$-module on
which $q\in P(\A_f)$ acts by multiplication by
$e^{\<d\chi,H_P(q)\>}$. A non-trivial class in a summand of
the right hand side is called a cohomology class \emph{of type}
$(\pi,w)$, $\pi\in\varphi_P$, $w\in W^P$. (This notion was first introduced in \cite{schwLNM}.)\\
Further, as $L(\R)\cong M(\R)\times T(\R)^\circ$,
$\widetilde\pi_\infty$ can be regarded as an irreducible, unitary
representation of $M(\R)$. Therefore, a $(\pi,w)$ type consists of
an irreducible representation $\pi=\chi\widetilde\pi$ whose unitary
part $\widetilde\pi=\widetilde\pi_\infty\hat\otimes\widetilde\pi_f$
(completed Hilbert tensor product) has at the infinite place an
irreducible, unitary representation $\widetilde\pi_\infty$ of the
semisimple group $M(\R)$ with non-trivial $(\m,K_M)$-cohomology with
respect to $F_w$.

\subsection{The possible archimedean components of $\pi$}
Recall that $\pi_\infty=\chi\wt\pi_\infty$, so classifying all
possible archimedean components $\pi_\infty$ of $\pi\in\varphi_P$
means to classify all characters $\chi$ of $A(\R)^\circ$ and
representations $\wt\pi_\infty$ of $M(\R)$ in question. \\ By the
above, we shall therefore find all irreducible unitary
representations of $M(\R)=SL_1(\H)\times SL_1(\H)$ which have
non-trivial $(\m,K_M)$-cohomology with respect to a representation
$F_w$, $w\in W^P$. This is achieved in the next

\begin{prop}\label{prop:cohcomp}
Let $V$ be an irreducible unitary representation of $M(\R)$ and $F$ any finite dimensional, irreducible representation of $M(\R)$. Then
$$H^q(\m,K_{M},V\otimes F)=\left\{
\begin{array}{ll}
\C & \textrm{if $q=0$ and $V\cong F$}\\
0 & \textrm{else}
\end{array}
\right.$$
\end{prop}

\begin{proof}
Since $M(\R)$ is compact, cohomology with respect
to $V\otimes F$ is one-dimensional,
if $V\cong \check{F}$ and $q=0$ and vanishes otherwise.
As $M(\C)$ is of Cartan type A$_1\times$ A$_1$, any finite dimensional representation $F$ of $M(\R)$ is self-dual.
\end{proof}

A direct calculation gives us furthermore the affine action of $W^P$ on $\lambda=\sum_{i=1}^3c_i\alpha_i$, restricted to $\b_\C$. It is listed in Table
\ref{t:muw}.

\begin{table}[h!]
\begin{center}
\begin{tabular}{r|c}
   & $\mu_w=w(\lambda+\rho)-\rho|_{\b_\C}$ \\ \hline
  $id$ & $(2c_1-c_2)\omega_1+(2c_3-c_2)\omega_3$ \\
  $w_2$ & $(c_1+c_2-c_3+1)\omega_1+(c_3+c_2-c_1+1)\omega_3$ \\
  $w_2w_1$ & $(2c_2-c_1-c_3)\omega_1+(c_1+c_3+2)\omega_3$ \\
  $w_2w_3$ & $(c_1+c_3+2)\omega_1+(2c_2-c_1-c_3)\omega_3$ \\
  $w_2w_1w_3$ & $(c_3+c_2-c_1+1)\omega_1+(c_1+c_2-c_3+1)\omega_3$ \\
  $w_2w_1w_3w_2$ & $(2c_3-c_2)\omega_1+(2c_1-c_2)\omega_3$ \\
\end{tabular}
\caption{The affine action of $W^P$}\label{t:muw}
\end{center}
\end{table}

We still need to classify the characters $\chi$ in question. As it follows from (\ref{dchi}), the character $\chi$ must
satisfy $d\chi = -w(\lambda+\rho)|_{\a_\C}$. Its possible values are hence given in Table \ref{t:dchis}.\\\\

\begin{table}[h!]
\begin{center}
\begin{tabular}{r|c}
   & $d\chi =-w(\lambda+\rho)|_{\a_\C}$ \\ \hline
  $id$ & $-(c_2+2)\alpha_2<0$ \\
  $w_2$ & $-(c_1-c_2+c_3+1)\alpha_2<0$ \\
  $w_2w_1$ & $-(c_3-c_1)\alpha_2\leq0$ \\\hline
  $w_2w_3$ & $(c_3-c_1)\alpha_2\geq0$ \\
  $w_2w_1w_3$ & $(c_1-c_2+c_3+1)\alpha_2>0$ \\
  $w_2w_1w_3w_2$ & $(c_2+2)\alpha_2>0$ \\
\end{tabular}
\caption{The possible characters $\chi$}\label{t:dchis}
\end{center}
\end{table}

\subsection{Definition of the Eisenstein map}
In order to construct Eisenstein cohomology classes, we start from a
class of type $(\pi,w)$. Since we are interested in cohomology, we
can assume without loss of generality that $\wt\pi_\infty=F_w$ (cf.,
Proposition \ref{prop:cohcomp}) and by (\ref{dchi}) that
$d\chi=-w(\lambda+\rho)|_{\a_\C}$. Moreover, we can assume that
$d\chi$ lies inside the closed, positive Weyl chamber defined by
$\Delta(P,A)$. So Table \ref{t:dchis} shows that we must have $w\in
W^+:=\{w_2w_3,w_2w_1w_3,w_2w_1w_3w_2\}$. \\\\ We reinterpret
$S_\chi(\t^*)$ as the (Banach) space of formal, finite $\C$-linear
combinations of differential operators $\frac{d^n}{d\Lambda^n}$ ($n$
being a multi-index $n=(n_1,n_2)$ with respect to a fixed coordinate
system) on $\t_\C$. It is a consequence of \cite{moewal}, Prop.
IV.1.11 or \cite{langbook} \S 7, that there is a non-trivial
polynomial $q(\Lambda)$ such that $q(\Lambda)E(f,\Lambda)$ is
holomorphic at $d\chi$ for all $K$-finite $f\in {\rm I}_{P,\wt\pi}$.
Since $\mathcal A_{\varphi}(G)$ can be written as the space which is
generated by the coefficient functions in the Taylor series
expansion of $q(\Lambda)E(f,\Lambda)$ at $d\chi$, $f$ running
through the $K$-finite functions in ${\rm I}_{P,\widetilde\pi}$, we
are able to define a surjective homomorphism $E_{\pi}$ of
$(\g_s,K,G(\A_f))$-modules between the $K$-finite elements in ${\rm
I}_{P,\widetilde\pi}\otimes S_\chi(\t^*)$ and $\mathcal
A_{\varphi}(G)$ by

\begin{equation}\label{eismap}
f\otimes\frac{d^n}{d\Lambda^n}\mapsto
\frac{d^n}{d\Lambda^n}\left(q(\Lambda)E(f,\Lambda)\right)|_{d\chi}.
\end{equation}
Therefore we get a well-defined map in cohomology

$$E^*_\pi:H^*(\g_s,K,{\rm I}_{P,\wt\pi}\otimes S_\chi(\t^*)\otimes
E)\longrightarrow H^*(\g_s,K,\mathcal A_{\varphi}(G)\otimes E)$$
which we will call {\it Eisenstein map}. It is this way, how we can lift
classes of type $(\pi,w)$ (which, as we now recall, are the elements of the space on the left hand side) to Eisenstein cohomology.

\subsection{Holomorphic Case}\label{holeis}
Suppose $[\omega]\in H^q(\g,K,{\rm I}_{P,\widetilde\pi}\otimes
S_\chi(\t^*)\otimes E)$ is a class of type $(\pi,w)$, represented by
a morphism $\omega$, such that for all elements
$f\otimes\frac{d^n}{d\Lambda^n}$ in its image,
$E_{\pi}(f\otimes\frac{d^n}{d\Lambda^n})=
\frac{d^n}{d\Lambda^n}\left(q(\Lambda)E(f,\Lambda)\right)|_{d\chi}$
is just the regular value $E(f,d\chi)$ of the Eisenstein series
$E(f,\Lambda)$, which is assumed to be holomorphic at the point
$d\chi=-w(\lambda+\rho)|_{\a_\C}$ inside the closed, positive Weyl
chamber defined by $\Delta(P,A)$. Then $E^q_\pi([\omega])$ is a
non-trivial Eisenstein cohomology class

$$E^q_\pi([\omega])\in H^q(\g_s,K,\mathcal A_{\varphi}(G)\otimes E).$$
This is a consequence of \cite{schwLNM}, Thm. 4.11.

\subsection{Residual Case}\label{reseis}
Suppose now that there is a $K$-finite $f\in {\rm I}_{P,\widetilde\pi}$ such that the Eisenstein series
$E(f,\Lambda)$ has a pole at $d\chi$ and notice that $E(f,\Lambda)$
is always holomorphic at $0$, by \cite{moewal}, Prop. IV.1.11 (b).
If $[\omega]\in H^q(\g_s,K,{\rm I}_{P,\widetilde\pi}\otimes
S_\chi(\t^*)\otimes E)$ is a class represented by a morphism
$\omega$ having only functions $f$ as in the previous sentence in
its image, then the residual Eisenstein cohomology class
$E^q_\pi\left([\omega]\right)$ defines a class

$$E^q_\pi\left([\omega]\right)\in H^{q'}(\g_s,K,\mathcal A_{\varphi}(G)\otimes E),$$
with $q'=4-q$. This follows from \cite{ich}, Thm. 2.1.\\\\
It is a rather delicate issue to describe the image of the
Eisenstein map $E^q_\pi$ in the residual case (in the above sense).
In order to do that, we need more knowledge on the residues of
Eisenstein series. Hence, we shall determine all relevant poles in
the next section.

\subsection{Poles of Eisensten series}

We have seen that the Eisenstein series we have to consider are
meromorphic functions in the parameter $\Lambda\in\t^*_\C$. However,
we have also seen that we only need to determine the behaviour of
holomorphy of the Eisenstein series at very certain points
$d\chi=-w(\lambda+\rho)|_{\a_\C^*}\in\a_\C^*\subset(\g_s)_{\C}^*$
given in Table \ref{t:dchis}. Thus, we would like to reduce the
problem of finding the poles of Eisenstein series to the level of
$G_s=SL_2'$. \\\\First we observe that writing
$\Lambda=(s_1,s_2)\in\t_\C^*$ with respect to the coordinates given
by the functionals $2\alpha$ and $2\beta$ (cf. section \ref{sec:1})
then the restriction of $\Lambda$ to $\a_\C^*$ is $s_1-s_2$ in the
coordinate given by the functional $2\alpha$. In terms of the
absolute roots of $\g_s$ this means that if
$\Lambda|_{\a_\C^*}=s\alpha_2$, then $s=s_1-s_2$. For our evaluation
points $d\chi$, these values $s$ are hence listed in Table
\ref{t:dchis}.\\\\ As a consequence of this consideration we only
need to check the behaviour of holomorphy of the Eisenstein series
along the line $s_1+s_2=0$ representing $\a_\C^*$ inside $\t_\C^*$,
a point $\Lambda$ in $\a_\C^*$ being identified with its coordinate
$s$ as above. For this purpose, let us recall the following result,
cf. \cite{moewal}, I.4.10:

\begin{prop}
The poles of the Eisenstein series $E(f,\Lambda)$ are the ones of
its constant Fourier coefficient $E(f,\Lambda)_P$ along $P$.
\end{prop}
Hence, we analyze this constant Fourier coefficient. It can be
written as
\begin{equation}\label{eq:constterm}
E(f,\Lambda)_P=f\cdot e^{\<\Lambda,H_{P}(.)\>}
+M(s,\widetilde\pi)(f\cdot e^{\<\Lambda,H_{P}(.)\>}),
\end{equation}
where $M(s,\widetilde\pi)$ is the meromorphic function in the parameter
$s\leftrightarrow\Lambda$, given for $g\in G(\A)$, $w$ the only
non-trivial element in $W_\Q$ and $Re(s)\gg 0$ by

$$M(s,\wt\pi):\textrm{Ind}_{P(\A)}^{G(\A)}[\wt\pi\otimes e^{\<\Lambda,H_P(.)\>}]\ra
  \textrm{Ind}_{P(\A)}^{G(\A)}[w(\wt\pi)\otimes
  e^{\<w(\Lambda),H_P(.)\>}]$$
  $$M(s,\widetilde\pi)\psi(g)=\int_{N(\A)}\psi(w^{-1}ng)dn.$$
as in \cite{moewal}, II.1.6. Since $L(\A)=GL_1'(\A)\times
GL_1'(\A)$, the cuspidal automorphic representation $\wt\pi$ can be
decomposed into $\wt\pi=\sigma\hat\otimes\tau$, $\sigma$ and $\tau$
being cuspidal automorphic representations of the factors
$GL_1'(\A)$. The action of $w$ on $\wt\pi$ then reads explicitly as
$w(\wt\pi)=\tau\hat\otimes\sigma$. Moreover, if $\Lambda\in\a^*_\C$
(which - as explained above - we shall always assume from now on)
then $w(\Lambda)=-\Lambda$.\\\\Now, let $S$ be a finite set of
places containing $S(D)$, such that $\widetilde\pi_p$ has got a
non-trivial $L(\Z_p)$-fixed vector for $p\notin S$: That is, outside
$S$, $L$ splits and $\widetilde\pi_p$ is unramified. We can hence
formally write $f\cdot
e^{\<\Lambda,H_{P}(.)\>}=\otimes'_p\psi_p=:\psi$ (restricted tensor
product over all places), where $\psi_p$ is a suitably normalized,
$L(\Z_p)$-fixed function for $p\notin S$. Therefore,
$M(s,\widetilde\pi)(f\cdot e^{\<\Lambda,H_{P}(.)\>})$ factors as
$M(s,\widetilde\pi)\psi=\otimes'_p M(s,\widetilde\pi_p)\psi_p$.
Using the Gindikin-Karpelevich integral formula, as shown in
\cite{lang3}, p. 27 (see also \cite{shahidi2}, p.554), we see that -
again for suitably normalized, non-trivial $L(\Z_p)$-fixed functions
$\widetilde\psi_p$

\begin{equation}\label{gindikin}
M(s,\wt\pi)\psi=\bigotimes_{p\in
S}M(s,\wt\pi_p)\psi_p\otimes\bigotimes'_{p\notin S}
\frac{L(s,\wt\pi_p,\check{r})}{L(1+s,\wt\pi_p,\check{r})}\widetilde
\psi_p.
\end{equation}
Here, $\check{r}$ is the dual of the adjoint representation of the
L-group of $L$ on the Lie algebra of the L-group of $N$ (see \cite{borLfun}, 2 and 3.4). It is
irreducible (\cite{lang3} sec. 6, case $(iii)$) and the
corresponding local $L$-function associated to $\wt\pi$ and
$\check{r}$ at the place $p\notin S$ is denoted
$L(s,\wt\pi_p,\check{r})$, cf. again \cite{borLfun}, 7.2. The following proposition is crucial.

\begin{prop}\label{prop:norm}
For all $p\in S(D)$ the operator $M(s,\wt\pi_p)$ is holomorphic and
non-vanishing in the region $Re(s)>0$. Hence, there is a $K$-finite
$f\in {\rm I}_{P,\wt\pi}$ such that the Eisenstein series
$E(f,\Lambda)$ has a pole at $\Lambda=s\alpha_2$, $Re(s)>0$, if and
only if the product $\prod_{p\notin S(D)}M(s,\wt\pi_p)$ has a pole
at $s$, $Re(s)>0$.
\end{prop}
\begin{proof}
Observe that for $p\in S(D)$, $\wt\pi_p$ is obviously supercuspidal,
(i.e. every matrix coefficient is compactly modulo the center $T(\Q_p)$ of $L(\Q_p)$), since for these
places $p$ the quotient $L(\Q_p)/T(\Q_p)$ is compact itself. So, for any such $p$,
$M(s,\widetilde\pi_p)$ is nothing but the intertwining operator whose image is the
Langlands quotient associate to $P(\Q_p)$, the tempered
representation $\wt\pi_p$ and the value $s$ with $Re(s)>0$. In particular,
$M(s,\widetilde\pi_p)$ is holomorphic and non-vanishing for $Re(s)>0$, see, e.g., \cite{bowa}, IV, Lemma 4.4 and XI, Cor. 2.7. Hence, the proposition follows.
\end{proof}
We distinguish three cases: Either (i) $\sigma$ and $\tau$ are both
not one-dimensional, (ii) exactly one of the representations
$\sigma$ and $\tau$ is one-dimensional, or (iii) $\wt\pi=\sigma\tau$
is a character of $L(\A)$.\\\\ Case (i) is the generic one, i.e.,
$\sigma_p$ and $\tau_p$ are both locally generic for $p\notin S(D)$.
This is well-known and can be seen as follows: For a cuspidal
automorphic representation $\rho$ of $GL_1'(\A)$, let $JL(\rho)$ be
the global Jacquet-Langlands lift as described in
\cite{jaclan} and \cite{gelbjacquet}. It is known that $JL(\rho)$ is an automorphic
representation of $GL_2(\A)$ appearing in the discrete spectrum.
Moreover, $JL(\rho)$ is cuspidal if and only if $\rho$ is not
one-dimensional, see \cite{gelbjacquet} Thm. 8.3. At $p\notin S(D)$
the Levi subgroup $L$ splits, i.e., $L(\Q_p)=GL_2(\Q_p)$ and the
local component of $JL(\rho)$ at such a $p$ satisfies the
(meaningful) equality $JL(\rho)_p=\rho_p$. So, assuming that $\rho$
is not a character, $\rho_p=JL(\rho)_p$ is the local component of a
cuspidal automorphic representation of $GL_2(\A)$. By
\cite{shalika}, corollary on p. 190, $\rho_p$ is hence
generic.\\
Coming back to the setup of case (i), i.e., $\sigma$ and $\tau$ are
both not one-dimensional, we see that for $p\notin S$,                      
$$L(s,\wt\pi_p,\check r)=L(s,\sigma_p\times\check\tau_p),$$
the usual local Rankin-Selberg $L$-function, cf. \cite{jac}. As a local Rankin-Selberg $L$-function has no poles and zeros in the region $Re(s)>1$, $L(s,\sigma_p\times\check\tau_p)^{-1}M(s,\wt\pi_p)$ is
holomorphic and non-vanishing in the region $Re(s)>0$ for $p\notin S$, cf. \eqref{gindikin}. In
\cite{moewalgln}, Prop. I.10, p.639, it is shown that for $p\in
S-S(D)$, $L(s,\sigma_p\times\check\tau_p)^{-1}M(s,\wt\pi_p)$ is
holomorphic and non-vanishing in the region $Re(s)>0$. Hence, the poles of the product $\prod_{p\notin S(D)}M(s,\wt\pi_p)$ are the ones of the partial $L$-function $\prod_{p\notin S(D)}L(s,\sigma_p\times\check\tau_p)$. Furthermore,
by \cite{gelbjacquet} Thm. 8.3 the local components $JL(\sigma)_p$
and $JL(\check{\tau})_p$ are both square-integrable for all $p\in
S(D)$, and therefore $L(s, JL(\sigma)_p\times JL(\check{\tau})_p)$ is
holomorphic and non-zero for $Re(s)>0$ and all $p\in S(D)$. This implies that the poles of $\prod_{p\notin S(D)}M(s,\wt\pi_p)$ in the region $Re(s)>0$ are the poles of the global Rankin-Selberg $L$-function $L(s,JL(\sigma)\times JL(\check{\tau}))$. Combining this with Proposition \ref{prop:norm}, we finally see that
in case (i), there is a ($K$-finite)
$f\in\textrm{I}_{P,\wt\pi}$ such that $E(f,\Lambda)$ has a pole at
$\Lambda=s\alpha_2$ with $Re(s)>0$, if and only if
$L(s,JL(\sigma)\times JL(\check{\tau}))$ has a pole. By the well--known analytic properties of global Rankin-Selberg $L$-functions, this happens if
and only if $s=1$ and $JL(\sigma)=JL(\tau)$, i.e., by strong
multiplicity one for $GL_1'$, if $s=1$ and $\sigma=\tau$.\\\\ For
the remaining cases, we use the work of N. Grbac. He normalized the
local intertwining operators for $p\notin S$ in \cite{nevenGn}, Cor.
2.2.5 using the work of C. M\oe glin and J.-L. Waldspurger
(\cite{moewalgln}). Using Grabc's result and a case-by-case analysis distinguishing the cases $p\in S(D)$, $p\in S-S(D)$ and $p\notin S$ analogous to the reasoning we provided above, it turns out that in case (ii) the poles of the
intertwining operator are the ones of the standard Langlands
$L$-function $L(s-\frac{1}{2},\sigma\check\tau)$ attached to the
representation $\sigma\check\tau$ of $GL_1'(\A)$. But this
$L$-function is entire, see \cite{jaclan}, so there are no poles in
case (ii). \\ If $\wt\pi=\sigma\tau$ is a character of $L(\A)$, then
the poles of an Eisenstein series are determined by the product

\begin{equation}\label{poles:char}
L(s,\sigma\tau^{-1})L(s-1,\sigma\tau^{-1})\prod_{p\in
S(D)}L(s-1,\sigma_p\tau^{-1}_p)^{-1},
\end{equation}
of Hecke $L$-functions. This was proved again in \cite{nevenGn}, Cor. 2.2.5, applying the
idea of \cite{moewalgln}, Lemme I.8, i.e., via induction from
generic representations of smaller parabolic subgroups. We therefore conclude that in case
(iii) $M(s,\wt\pi)$ has a pole at $Re(s)>0$ if and only if $s=2$ and
$\sigma=\tau$. To see this, observe that the poles of the global
Hecke $L$-functions $L(s,\sigma\tau^{-1})L(s-1,\sigma\tau^{-1})$ at $s=1$ are canceled by the zeros of the inverse
of the $|S(D)|$-many local $L$-functions. As $D$ is non-split over
$\Q$, $|S(D)|\geq 2$.\\\\ We summarize the results of this section
in the following proposition.

\begin{prop}\label{prop:poles}
Let $\wt\pi=\sigma\otimes\tau$ be a unitary cuspidal automorphic
representation of $L(\A)$. There is a $K$-finite function
$f\in\textrm{I}_{P,\wt\pi}$, such that $E(f,\Lambda)$ has a pole at
$\Lambda=s\alpha_2$, $Re(s)>0$, if and only if $\sigma=\tau$ and
either

\begin{enumerate}
 \item $\dim\sigma>1$ and $s=1$ or
 \item $\dim\sigma=1$ and $s=2$
\end{enumerate}

\end{prop}

\subsection{The image of the Eisenstein map revisited}

In sections \ref{holeis} and \ref{reseis} we gave a (still
incomplete) description of the image (and so {\it vice versa} of the
kernel) of the Eisenstein map $E^q_\pi$. In order to complete it, we
still need to understand the image of $E_\pi^q$ in the case of
non-holomorphic Eisenstein series (see section \ref{reseis}).
Therefore observe that the residue of an Eisenstein series, as
determined in Proposition \ref{prop:poles}, generates a residual
automorphic representation of $G(\A)$, i.e., an irreducible
subrepresentation of $L^2_{res}(G(\Q)Z(\R)^\circ\backslash G(\A))$,
cf. \ref{sec:Eisseries}. In particular, $Z(\R)^\circ$ acts trivially
on such a representation. As $G(\R)=Z(\R)^\circ\times G_s(\R)$, cf.
section \ref{sec:realgrps}, we may view its archimedean component as
an irreducible, unitary representation of $G_s(\R)=SL_2(\H)$.
Clearly, we are only interested in residual automorphic
representations, which have non-vanishing $(\g_s,K)$-cohomology
tensorised by the $G_s(\R)$-representation $E$. Therefore, in order
to understand the image of the Eisenstein map in the case of
non-holomorphic Eisenstein series, we have to understand the
cohomological unitary dual of $G_s(\R)$. It is classified in the
next proposition.

\begin{prop}\label{theoremcohunit}
For each irreducible, finite-dimensional representation $E$ of
$G_s(\R)$ of highest weight $\lambda=\sum_{i=1}^3c_i\alpha_i$ there is an integer
$j(\lambda)$, $0\leq j(\lambda)\leq 3$ such that the irreducible,
unitary representations of $G_s(\R)$ with non-trivial cohomology
with respect to $E$ are the uniquely determined representations
$A_{j}(\lambda)$, $j(\lambda)\leq j\leq 2$ having the property

$$H^q(\g,K,A_{j}(\lambda)\otimes E)=\left\{
\begin{array}{ll}
\C & \textrm{if $q=j$ or $q=5-j$}\\
0 & \textrm{otherwise}
\end{array}
\right.$$
This integer is given as
$$j(\lambda)=\left\{
\begin{array}{ll}
0 & \textrm{if $\lambda=0$}\\
1 & \textrm{if $\lambda=k\omega_2, k\geq1$ }\\
2 & \textrm{if $\lambda\neq k\omega_2, k\geq0$ but $c_1=c_3$}\\
3 & \textrm{otherwise}
\end{array}
\right.$$ Let $J(F,t)$ be the Langlands' Quotient associate to the
triple $(P_s(\R),F,t\alpha_2)$, $F$ an irreducible representation
of $M(\R)$ and $t>0$. Then these representations read as follows:
$$A_0(\lambda)=\triv_{G_s(\R)}=J(F_{id},2)$$
$$A_1(\lambda)=J(F_{w_2},1)$$
$$A_2(\lambda)=\textrm{\emph{Ind}}^{G_s(\R)}_{P_s(\R)}[F_{w_2w_3}\otimes\triv_{A(\R)}]$$
\end{prop}
\begin{proof}[Short sketch of a proof]
This can be achieved using the well-known Vogan-Zuckerman
classification of the cohomological unitary dual of connected
semisimple Lie groups. More precisely it is a consequence of Thm.s
5.5, 5.6, 6.16 and Prop. 6.5 in \cite{vozu}. The information relevant for applying these results to our specific case consists of a list of the so-called $\theta$-stable parabolic subalgebras $\q$ of $(\g_s)_\C\cong\mathfrak{sl}_4(\C)$ up to $K$-conjugacy. There are four such classes, represented by subalgebras $\q_0$, $\q_1$, $\q_2$ and $\q_3$. In fact, the esssential data (cf. Table \ref{t:qs}) is already provided by knowing the Levi subalgebaras $\l_j$, $j=0,1,2,3$, of these parabolic subalgebras and the sets $\Delta_j$ of those roots which appear both, in the direct sum decomposition of the nilpotent radical of $\q_j$, and the $(-1)$-Eigenspace of $\theta$ in $\g_s$. If two such sets of roots coincide, the corresponding $\theta$-stable parabolic subalgebras, although not $K$-conjugate, provide isomorphic irreducible unitary representations and may hence be identified. According to Table \ref{t:qs}, we may hence focus on $\l_0$, $\l_1$ and $\l_2$. Their corresponding irreducible, unitary representations are $A_0(\lambda)$, $A_1(\lambda)$ and $A_2(\lambda)$ as described in our proposition.
\begin{table}[h!]
\begin{center}
\begin{tabular}{r|c|c}
  & $\l_j$ & $(-i)\Delta_j$ \\ \hline
  $j=0$ & $\g_s$ & $\emptyset$ \\
  $j=1$ & $\R\oplus\mathfrak{sl}_2(\C)$ & $\{\varepsilon_1+\varepsilon_3\}$ \\
  $j=2$ & $\R\oplus\C$ & $\{\varepsilon_1+\varepsilon_3,\varepsilon_1+\varepsilon_4\}$ \\
  $j=3$ & $\R^2\oplus\mathfrak{sl}_2(\R)$ & $\{\varepsilon_1+\varepsilon_3,\varepsilon_1+\varepsilon_4\}$ \\
\end{tabular}
\caption{}\label{t:qs}
\end{center}
\end{table}
\end{proof}

\begin{rem}\label{rem:temp}
If $c_1\neq c_3$, i.e., $E\ncong\check{E}$, then the proposition
says that there are no irreducible, unitary cohomological
representations of $G_s(\R)$. \\As it follows from the last paragraph
of \cite{vozu}, p.8, $A_2(\lambda)$ is the only tempered
representation among the $A_i(\lambda)$, $i=0,1,2$.
\end{rem}

Now, let $(\pi,w)$ be a tuple as before, i.e., $w\in W^P$ and
$\pi=\chi\wt\pi\in\varphi_P$ with $\wt\pi=\sigma\otimes\tau$ and
$d\chi=-w(\lambda+\rho)|_{\a_\C}=s\alpha_2$, $Re(s)>0$. Assume that
$\pi$ gives rise to Eisenstein series $E(f,\Lambda)$ which have a
pole at the uniquely determined point $\Lambda=d\chi$, i.e.,
$\wt\pi$ and $s$ satisfy the necessary and sufficient conditions of
Proposition \ref{prop:poles}. Let $\Pi$ be the residual automorphic
representation spanned by the residues of these Eisenstein series.
We know from Proposition \ref{theoremcohunit} that its archimedean
component $\Pi_\infty$ must be one of the representations
$A_i(\lambda)$, $i=0,1,2$ (to be of interest for us). In fact, we
may exclude the case $i=2$. In order to see this, recall from remark
\ref{rem:temp} that $A_2(\lambda)$ is tempered. But as it is shown
in \cite{wallach} Thm. 4.3, any automorphic representation which
appears in the discrete spectrum and has a tempered archimedean
component can only appear in the cuspidal spectrum. But this would
contradict the assumption that $\Pi$ is residual.\\  In the
remaining possible cases, $i=0,1$, the irreducible unitary
representation $\Pi_\infty$ is a proper Langlands quotient
$J(F_w,t)$, $w\in W^P$, $t=1,2$. According to Proposition
\ref{theoremcohunit} the Kostant representative $w$ and the number
$t$ are given es follows: If $\Pi_\infty = A_i(\lambda)$, then
$w=w_i$ is the unique element of $W^P$ of length $\ell(w)=i$ and
$t=2-i$. Therefore, for $i=0,1$ we have an exact sequence
$$0\ra U_i(\lambda)\ra {\rm Ind}_{P_s(\R)}^{G_s(\R)}[F_{w_i}\otimes\C_{(2-i)\alpha_2}]\ra A_i(\lambda)\ra 0$$
for a certain representation $U_i(\lambda)$. It is clear that the
only subquotients of $U_i(\lambda)$ can be $A_1(\lambda)$ or
$A_2(\lambda)$. Since the $(\g_s,K)$-cohomology of $A_j(\lambda)$,
$j=1,2$, is non-trivial in degree $j$, also ${\rm
Hom}_K(\Lambda^j(\g_s/\k),A_j(\lambda)\otimes E)\neq 0$. But, as all
$K$-isotypic components of ${\rm
Ind}_{P_s(\R)}^{G_s(\R)}[F_{w_j}\otimes\C_{(2-j)\alpha_2}]$ are of
multiplicity one, we get that $U_i(\lambda)=A_{i+1}(\lambda)$. 
This enabels us to show the following result which completes the
description of the image of the Eisenstein map.

\begin{thm}\label{thm:imageeis}
Let $(\pi,w)$ be a pair as above and suppose
$\Pi_\infty=A_i(\lambda)$, $i=0,1$. Let $\Omega_{res}(s,\wt\pi)$ be
the span of those classes $[\omega]$ of type $(\pi,w)$ such that the
associated Eisenstein series have a pole at the uniquely determined
point $\Lambda=s\alpha$. Then the image of $\Omega_{res}(s,\wt\pi)$
under the Eisenstein map $E^q_\pi$ is non-trivial if and only if
$q=4-i$.
\end{thm}
\begin{proof}
Recall from the above that we have an exact sequence 
$$0\ra A_{i+1}(\lambda)\ra {\rm Ind}_{P_s(\R)}^{G_s(\R)}[F_{w_i}\otimes\C_{(2-i)\alpha_2}]\ra A_i(\lambda)\ra 0.$$
Hence, we are in the situation considered in \cite{rospSO}, Lem. 1.4.1. and so formally the same arguments as presented in the proof of \cite{rospSO}, Prop. 1.4.3 show that $E^q_\pi(\Omega_{res}(s,\wt\pi))=0$ if $q'\neq i$, cf. section \ref{reseis}. We therefore obatin that the Eisenstein
map can only have non-trivial image if $q=4-q'=4-i$. However, in this degree the $G(\A_f)$-module $E^q_\pi(\Omega_{res}(s,\wt\pi))$
is really non-zero. This follows either by the arguments given in \cite{rospSO}, Thm. 1.4.4 or -- more directly -- if we observe that the following conditions are matched in our specific case: The pole of the Eisenstein series we are looking at is obtained as a pole of the intertwining operator $M(s,\wt\pi)$ at a point $s>0$. Since $G_s$ is of $\Q$-rank $1$, this pole is automatically of maximal possible order (namely $1$, cf. \cite{moewal} IV. 1.11), $w$ is clearly the longest element in the quotient $W(A):=N_{G_s(\Q)}(A(\Q))/L_s(\Q)\cong\{id,w\}$ and the corresponding representation $\Pi$ spanned by the residues of the Eisenstein series is square-integrable. Furthermore, the archimedean component $\Pi_\infty$ is the image of $M(s,\wt\pi_\infty)$, which is the Langlands quotient associate to the real parabolic subgroup $P(\R)$, the tempered representation $\wt\pi_\infty$ and the value $s>0$, cf. the proof of Proposition \ref{prop:norm}. The minimal degree in which $\Pi_\infty=A_i(\lambda)$ for $i\in\{0,1\}$ has non-zero $(\g_s,K)$-cohomology is $q'=i$, cf. Proposition \ref{theoremcohunit}. Hence, all assumptions made in \cite{rosp}, Thm. III.1 are satisfied in our specific case and the non--triviality of $E^q_\pi(\Omega_{res}(s,\wt\pi))$ is a consequence of \cite{rosp}, Thm. III.1.
\end{proof}

\subsection{Determination of Eisenstein cohomology}

Now we are ready to prove our first main theorem.

\begin{thm}\label{theoremeis}
Let $G=GL_2'$ and $E$ be any finite-dimensional, irreducible,
complex-rational representation of $G(\R)$ of highest weight
$\lambda=\sum_{i=1}^3c_i\alpha_i$ and assume that $Z(\R)^\circ$ acts trivially
on $E$. For any tuple $(\pi,w)$, $w\in W^P$ and
$\pi=\chi\wt\pi\in\varphi_P$ with $\wt\pi=\sigma\otimes\tau$ and
$d\chi=-w(\lambda+\rho)|_{\a_\C}=s\alpha_2$, let
$\Omega_{hol}(s,\wt\pi)$ \emph{(}resp.
$\Omega_{res}(s,\wt\pi)$\emph{)} be the span of those classes
$[\omega]$ of type $(\pi,w)$ such that the associated Eisenstein
series are holomorphic \emph{(}resp. have a pole\emph{)} at the
uniquely determined point $\Lambda=s\alpha$, $Re(s)\geq 0$. Then the Eisenstein
cohomology of $G$ with respect to $E$ is given as follows:\\\\
\textrm{\emph{(1)} \underline{If $\lambda=k\omega_{2}$, $k\in\Z_{\geq
0}$}}:

\begin{eqnarray*}
H^0_{Eis}(G,\C) & = &  \bigoplus_{\substack{d\chi=2\alpha_2\\
 \wt\pi_\infty=\triv_{M(\R)}\\ \sigma=\tau}}
E_{\pi}^4(\Omega_{res}(2,\pi))\quad \textrm{if $k=0$}\\
& = & \bigoplus_{\substack{\sigma\\ \sigma_\infty=\triv_{SL_1(\H)}}} \sigma_f\circ\textrm{\emph{det}}'\\
H^1_{Eis}(G,E) & = &
 \bigoplus_{\substack{d\chi=\alpha_2\\
 \wt\pi_\infty=
 F_{w_2w_1w_3}\\ \sigma=\tau}}
 E^3_\pi(\Omega_{res}(1,\pi))\\
H^2_{Eis}(G,E) & = &
 \bigoplus_{\substack{d\chi=0\\
 \pi_\infty=F_{w_2w_3}}}\textrm{\emph{Ind}}^{G(\A_f)}_{P(\A_f)}[\pi_f]\\
H^3_{Eis}(G,E) & = &
 \bigoplus_{\substack{d\chi=\alpha_2\\
 \wt\pi_\infty=F_{w_2w_1w_3}\\
 \sigma\neq\tau}}\textrm{\emph{Ind}}^{G(\A_f)}_{P(\A_f)}[\C_{\alpha_2}\otimes
 \pi_f]\oplus\\
& & \bigoplus_{\substack{d\chi=\alpha_2\\
 \wt\pi_\infty=F_{w_2w_1w_3}\\
  \sigma=\tau}}\Omega_{hol}(1,\pi)\\
H^4_{Eis}(G,\C) & = &
 \bigoplus_{\substack{d\chi=2\alpha_2\\
 \wt\pi_\infty=\triv_{M(\R)} \textrm{ \emph{but} }\\
 \sigma\neq\tau}}\textrm{\emph{Ind}}^{G(\A_f)}_{P(\A_f)}[\C_{2\alpha_2}\otimes
 \pi_f]\oplus\\
 & & \bigoplus_{\substack{d\chi=2\alpha_2\\
 \wt\pi_\infty=\triv_{M(\R)}\\ \sigma=\tau}} \Omega_{hol}(2,\pi)\quad \textrm{if $k=0$}\\
H^4_{Eis}(G,E) & = &
 \bigoplus_{\substack{d\chi=(c_2+2)\alpha_2\\
 \wt\pi_\infty=\triv_{M(\R)}}}\textrm{\emph{Ind}}^{G(\A_f)}_{P(\A_f)}[\C_{(c_2+2)\alpha_2}\otimes
 \pi_f]\quad \textrm{if $k\neq0$}\\
H^q_{Eis}(G,E) & = & 0 \quad\textrm{ else.}
\end{eqnarray*}
Cohomology in degrees $2,3$ and $4$ is entirely built up by values
of holomorphic Eisenstein series. Cohomology in degree $0$ and $1$
consists of residual classes, which can be represented by
square-integrable, residual automorphic forms. Both spaces do not vanish.\\\\
\textrm{\emph{(2)} \underline{If $\lambda\neq k\om_{2}$, $k\in\Z_{\geq
0}$}}:

\begin{eqnarray*}
H^2_{Eis}(G,E) & = &
 \bigoplus_{\substack{d\chi=0\\
 \pi_\infty=F_{w_2w_3}}}\textrm{\emph{Ind}}^{G(\A_f)}_{P(\A_f)}[\pi_f]\quad \textrm{if $c_1=c_3$}\\
H^3_{Eis}(G,E) & = &
 \bigoplus_{\substack{d\chi=(c_1-c_2+c_3)\alpha_2\\
 \wt\pi_\infty=F_{w_2w_1w_3}}}\textrm{\emph{Ind}}^{G(\A_f)}_{P(\A_f)}[\C_{(c_1-c_2+c_3)\alpha_2}\otimes
 \pi_f]\\
 \end{eqnarray*}
\begin{eqnarray*}
H^4_{Eis}(G,E) & = &
 \bigoplus_{\substack{d\chi=(c_2+2)\alpha_2\\
 \wt\pi_\infty=F_{w_2w_1w_3w_2}}}\textrm{\emph{Ind}}^{G(\A_f)}_{P(\A_f)}[\C_{(c_2+2)\alpha_2}\otimes
 \pi_f]\\
H^q_{Eis}(G,E) & = & 0 \quad\textrm{ else.}
\end{eqnarray*}
All of these spaces are entirely built up by values of holomorphic
Eisenstein series, whence the are no residual Eisenstein cohomology
classes in this case.
\end{thm}

\begin{proof}
By \cite{boserre}, Cor. 11.4.3, $H^q(G,E)=0$ if $q\geq\dim_\R X=5$.
So we may concentrate on degrees $0\leq q\leq 4$. Let us first
consider the case $q=0,4$. Take $\pi=\chi\wt\pi\in\varphi_P$. In
order to be of cohomological interest, it must satisfy
$\wt\pi_\infty=F_{w}$, with $w=w_2w_1w_3w_2\in W^P$ (i.e.,
$\wt\pi_\infty$ is the irreducible finite-dimensional representation
of $M(\R)$ of highest weight
$\mu_{w}=(2c_3-c_2)\omega_1+(2c_1-c_2)\omega_3$) and
$d\chi=(c_2+2)\alpha_2$. This follows from Proposition
\ref{prop:cohcomp} and Tables \ref{t:muw} and \ref{t:dchis}. In
particular, $s=c_2+2\geq 2$. So, in order to get a pole of an
Eisenstein series associated to $\pi$ at $s$, we know from
Proposition \ref{prop:poles} that it is necessary and sufficient
that $\wt\pi$ is a character of the form
$\wt\pi=\sigma\otimes\sigma$ and $c_2=0$. It follows that
$$1=\dim\sigma_\infty=2c_3-c_2+1=\dim\tau_\infty=2c_1-c_2+1,$$
i.e., $c_1=c_3=0$, too, whence $\lambda=0$ and
$\wt\pi_\infty=\triv_{M(\R)}=F_{id}$. Recalling that under this
assumption $s=2$, Theorem \ref{thm:imageeis} together with section
\ref{holeis} show the assertion in degrees $q=0,4$.\\ In degrees
$q=1,3$ we have to have $\wt\pi_\infty=F_w$, with $w=w_2w_1w_3$ and
$d\chi=(c_1-c_2+c_3+1)\alpha_2$. Checking with Table \ref{t:muw}
yields $\dim\sigma_\infty\geq2$ and $\dim\tau_\infty\geq 2$, so both
factors of $\wt\pi$ are infinite-dimensional. In order to give rise
to a pole, our evaluation point $s$ must therefore be equal to $s=1$
(cf., Proposition \ref{prop:poles}), which implies $c_1-c_2+c_3=0$,
i.e., $\lambda=k\omega_2$, $k\in\Z_{\geq 0}$. Further, $\sigma$ must
be $\tau$, whence the equation $c_1=c_3$. Finally, the space
$E^3_\pi(\Omega_{res}(1,\pi))$ is non-trivial because of Theorem
\ref{thm:imageeis} which together with section \ref{holeis} shows
the assertion in degrees $q=1,3$.\\ Finally, we consider degree
$q=2$. Suppose there is a pole of an Eisenstein series at the
uniquely determined point $s=c_3-c_1\geq 0$, cf. Table
\ref{t:dchis}, for a cuspidal representation $\pi\in\varphi_P$. As
$\dim\sigma_\infty=c_1+c_3+3\geq 3$, Proposition \ref{prop:poles}
tells us that $c_3-c_1=1$ and $\sigma=\tau$. But this implies
$c_1+c_3+2=2c_2-c_1-c_3$ (cf., Table \ref{t:muw}), leading to
$2=c_2-2c_1$. As $2c_1\geq c_2$ we end up in a contradiction. Hence,
all Eisenstein series contributing to degree $q=2$ are holomorphic
at the evaluation point $s=c_3-c_1$. Still, the archimedean
component of $\textrm{Ind}^{G(\A)}_{P(\A)}[\pi]$ is a {\it unitary}
representation of $G_s(\R)$. Therefore $H^2_{Eis}(G,E)=0$ for all
representations $E$, which are not self-contragredient by
Proposition \ref{theoremcohunit}. Note that $E\cong\check{E}$ is
equivalent to $c_1= c_3$, cf. remark \ref{rem:temp}, or otherwise
put, $s=0$. Now, the proof of the theorem is complete.
\end{proof}

\subsection{Langlands Functoriality and Eisenstein cohomology in degree $q=1$}
We want to point out that the global Jacquet-Langlands
Correspondence for $GL_n'$, as it was recently established by I. A.
Badulescu and D. Renard in \cite{ioan}, gives an alternative proof
of the non-vanishing of $H^1_{Eis}(G,E)$:

\begin{thm}\label{thm:q=1}
For all $E$ with highest weight $\lambda=k\omega_2$, $k\geq 0$,
there is a residual automorphic representation $\pi$ of $G(\A)$
which has cohomology in degree $q=1$. In particular,
$H^1_{Eis}(G,E)\neq 0$.
\end{thm}
The philosophy of the following proof will be to use the classical
Jacquet-Langlands Correspondence for $GL_1'$ to construct an
appropriate cuspidal automorphic representation of $GL_2(\A)\times
GL_2(\A)$ and then take the unique irreducible quotient of the
representation induced to $GL_4(\A)$. Using the work of Badulescu
and Renard we can therefrom construct a residual representation
$\pi$ of $G(\A)$ having the claimed properties. As this was already shown in Theorem \ref{theoremeis}, we allow ourselves to keep the proof of this fact rather short by assuming some familiarity with the paper \cite{ioan}.\\ Its underlying idea fits very well
with the idea to use functoriality in order to get cohomological
automorphic representations. The interested reader may find a survey
on this topic in \cite{ragsha}, section 5.2.
\begin{proof}
By Table \ref{t:muw},
$F_{w_2}=\textrm{Sym}^{k+1}\C^2\otimes\textrm{Sym}^{k+1}\C^2$.
Obviously, $\textrm{Sym}^{k+1}\C^2$ is a discrete series
representation of $GL_1'(\R)=GL_1(\H)$. So there is a cuspidal
automorphic representation $\rho'$ of $GL_1'(\A)$ having archimedean
component $\rho'_\infty=\textrm{Sym}^{k+1}\C^2$. The classical
Jacquet-Langlands lift, \cite{gelbjacquet} Thm. 8.3, gives us therefore a
cuspidal automorphic representation $\rho:=JL(\rho')$ of $GL_2(\A)$
having only square-integrable representations $\rho_v$ at all places
$v\in S(D)$. By the description of the residual spectrum of the $\Q$-split group
$GL_n$, cf. \cite{moewalgln}, the unique irreducible quotient
$\sigma$ of Ind$^{GL_4(\A)}_{GL_2(\A)\times
GL_2(\A)}[|\det|^{\frac{1}{2}}\rho\otimes
|\det|^{-\frac{1}{2}}\rho]$ is a residual representation of
$GL_4(\A)$. Now, we see by Prop. 15.3.(a) of \cite{ioan} that for
each $v\in S(D)$, $\sigma_v$ is ``$2$-compatible'', i.e., the
inverse of the local Jacquet-Langlands correspondence $JL^{-1}_v$
from the level of $GL_4(\Q_v)$ to the level of $GL'_2(\Q_v)$ is
non-trivial. Here we use that $\rho_v$ is square-integrable at all
places $v\in S(D)$. Hence, we might apply Thm. 18.1 of \cite{ioan}
and see that $\sigma$ is in the image of the global Jacquet-Langlands
correspondence developed in the paper \cite{ioan}. It follows that
there is an unique representation $\pi$ of $G(\A)$ which appears in
the discrete spectrum of $G$ and corresponds to $\sigma$. It is
residual by Prop. 18.2.(b) of \cite{ioan}. By its very construction,
the archimedean component of it satisfies
$\pi_\infty|_{G_s(\R)}=J(F_{w_2},1)$, cf. \cite{ioan}, Thm. 13.8.
Combining this with our Proposition \ref{theoremcohunit} and
\cite{rosp}, Thm. III.1, we get the claim.
\end{proof}

\subsection{Revisiting Eisenstein cohomology in degree $q=2$}
For our special case $G=GL_2'$, we would now like to take up a question once raised by G. Harder.
As the contents of this section will not be needed in the sequel, we allow ourselves to be rather brief.
Recall the adelic Borel-Serre Compactification $\overline X_\A$ of
$X_\A:=G(\Q)\backslash (X\times G(\A_f))$, and its basic properties:
(For this we refer to \cite{boserre} as the original source and
\cite{roh} for the adelic setting.) It is a compact space with
boundary $\partial(\overline X_\A)$ and the inclusion
$X_\A\hra\overline X_\A$ is a homotopy equivalence. Furthermore, there is the natural restriction morphism of
$G(\A_f)$-modules
$$res^q: H^q(X_\A,\wt E)\cong H^q(\overline X_\A,\wt E)\ra H^q(\partial(\overline X_\A),\wt E).$$
Here, $\wt E$ stands for the sheaf with \'espace \'etal\'e $(X\times G(\A_f))\times_{G(\Q)} E$, $E$ given the discrete topology. It is finally a
consequence of \cite{franke} Thm. 18 that there is also the following isomorphism of $G(\A_f)$-modules
$$H^q(\overline X_\A,\wt E)\cong H^q(G,E).$$
It makes therefore sense to talk about Eisenstein cohomology
as a subspace of $H^q(\overline X_\A,\wt E)$ and hence to restrict Eisenstein cohomology classes to the cohomology of the boundary $\partial(\overline X_\A)$.\\\\
Now, let $q=2$. We know that for each $\pi=\sigma\otimes\tau$

$$H^2(\g_s,K,\textrm{Ind}^{G(\A)}_{P(\A)}[\sigma\otimes\tau]\otimes E)\quad\textrm{and}\quad H^2(\g_s,K,\textrm{Ind}^{G(\A)}_{P(\A)}[\tau\otimes\sigma]\otimes E)$$
are linear independent subspaces of $H^2(\partial(\overline
X_\A),\wt E)$, since all Eisenstein series showing up in degree $2$
are holomorphic (cf. \cite{bor2}, Lemma 2.12). Taking the restriction $res^2([\omega])$ of a
class $[\omega]\in \textrm{Ind}^{G(\A_f)}_{P(\A_f)}[\pi_f]\subset
H^2_{Eis}(G,E)$ to the boundary means to calculate the constant term
of the Eisenstein series representing $\omega$ and then taking the
corresponding class (cf. \cite{schwLNM}, Satz 1.10). Therefore, by \eqref{eq:constterm}

\begin{eqnarray*}
res^2([\omega]) & = & [\omega]  \oplus  [M(0,\pi)_*\omega]\\
& \in & H^2(\g_s,K,\textrm{Ind}^{G(\A)}_{P(\A)}[\sigma\otimes\tau]\otimes E)  \oplus  H^2(\g_s,K,\textrm{Ind}^{G(\A)}_{P(\A)}[\tau\otimes\sigma]\otimes E)
\end{eqnarray*}
G. Harder asked (in a more general context), if $[M(0,\pi)_*\omega]\neq
0$ for some Eisenstein class $[\omega]$. This is actually true in
our case. We devote the next theorem to this result. Clearly, we
only need to check the archimedean place on the level of cohomology.

\begin{thm}\label{t:nonzero}
The local intertwining operator $M(0,\pi_\infty)$ is an isomorphism. In particular, it induces an
isomorphism of cohomologies
$$[M(0,\pi_\infty)_*]:H^2(\g_s,K,\textrm{\emph{Ind}}^{G_s(\R)}_{P_s(\R)}[\sigma_\infty\otimes\tau_\infty]\otimes E)
\ira
H^2(\g_s,K,\textrm{\emph{Ind}}^{G_s(\R)}_{P_s(\R)}[\tau_\infty\otimes\sigma_\infty]\otimes
E)$$
\end{thm}
\begin{proof}
The operator $M(0,\pi_\infty)$ is holomorphic, cf. \cite{moewal}, Prop. IV.1.11 (b). The same holds for its adjoint map
$M(0,\sigma_\infty\otimes\tau_\infty)^*=M(0,\tau_\infty\otimes\sigma_\infty)$.
Here we used \cite{arthur}, (J$_4$), p.26. {\it Ibidem}, line (3.5)
shows that the composition $M(0,\pi_\infty)^*M(0,\pi_\infty)$ is a
positive real number
$$M(0,\pi_\infty)^*M(0,\pi_\infty)=|r_{\overline P|P}(\pi_\infty)|^2>0.$$
It follows that $M(0,\pi_\infty)$ is not identically zero. Now
observe that the induced representations
$$\textrm{Ind}^{G_s(\R)}_{P_s(\R)}[F_{w_2w_3}\otimes
\triv_{A(\R)}]\quad\textrm{and}\quad
\textrm{Ind}^{G_s(\R)}_{P_s(\R)}[F_{w_2w_1}\otimes \triv_{A(\R)}]$$
are isomorphic since $\sigma_\infty\otimes\tau_\infty=F_{w_2w_3}$
and $\tau_\infty\otimes\sigma_\infty=F_{w_2w_1}$ are conjugate by an
element $m\in N_K(A(\R)^\circ)$ which is not in $M(\R)$. Therefore,
we can view $M(0,\pi_\infty)$ as an endomorphism
$$M(0,\pi_\infty):\textrm{Ind}^{G_s(\R)}_{P_s(\R)}[F_{w_2w_3}\otimes
\triv_{A(\R)}]\ra \textrm{Ind}^{G_s(\R)}_{P_s(\R)}[F_{w_2w_3}\otimes
\triv_{A(\R)}].$$ Recall that
$\textrm{Ind}^{G_s(\R)}_{P_s(\R)}[F_{w_2w_3}\otimes \triv_{A(\R)}]$
is irreducible. As $M(0,\pi_\infty)$ is not identically zero, it is
hence an isomorphism and the claim holds.
\end{proof}

Observe that $G_s$ is an inner form of the $\Q$-split group $SL_4/\Q$.
Harder writes in \cite{har_rankone}, 2.3.2, that he got a letter of
B. Speh in which she shows that Harder's operator
$T^{\textrm{loc}}_\infty(0)$ induces zero on the level of
cohomology. Here, $T^{\textrm{loc}}_\infty(0)$ is actually the
analog of our operator $M(0,\pi_\infty)$, connecting induced
representations of $L_1:=SL_1\times SL_3$ and $L_3:=SL_3\times SL_1$
sitting as Levi factors of the two non self-associate maximal
parabolic $\Q$-subgroups inside $SL_4$. Of course our
(self-associate) Levi subgroup $L_s$ is an inner form of the second
(and hence self-associate) Levi subgroup $L_2:=SL_2\times SL_2$ of
$SL_4$. We expect that Theorem \ref{t:nonzero} also holds for
$L_2$.\\\\ We also consider Theorem \ref{t:nonzero} as the starting point
for further investigations, whose aim is to establish the
rationality of critical values of $L$-functions, as it was carried
out by Harder in \cite{har_rankone} for the $\Q$-split group
$SL_3/\Q$. The proof of such a result would go beyond the scope of
this paper and is currently work in progress. We hope to report on it in a forthcoming work.

\section{Cuspidal Cohomology}

\subsection{}
Having analyzed Eisenstein cohomology in
the previous sections, we still need to describe the space of cuspidal cohomology
$$H^q_{cusp}(G,E)=H^q(\g_s,K,\mathcal A_{cusp}(G)\otimes E)$$
as defined in section \ref{sec:defcoh}, in order to know the full space $H^q(G,E)$.
Recall from (\ref{eq:cuspcohdec}) that we have the direct sum decomposition
\begin{equation*}
H^q_{cusp}(G,E)=\bigoplus_{\pi}H^q(\g_s,K,\pi\otimes E),
\end{equation*}
the sum ranging over all (equivalence classes of) cuspidal automorphic representations
$$\pi\hookrightarrow L_{cusp}^2(G(\Q)Z(\R)^\circ\backslash G(\A))$$
of $G(\A)$. In particular, $Z(\R)^\circ$ acts trivially
on such a representation. Since $G(\R)=Z(\R)^\circ\times G_s(\R)$, cf.
section \ref{sec:realgrps}, we may again view its archimedean component as
an irreducible, unitary representation of $G_s(\R)$. The
cohomological irreducible unitary representations of $G_s(\R)$ were
classified in the Proposition \ref{theoremcohunit}.\\\\
Let us start with the well-known fact that there cannot be any
non-trivial cuspidal cohomology class in degree $q=0$. There are many
different proofs. An easy way to see it, is that among all
irreducible admissible representations of $G(\A)$ only the trivial
representation $\triv=\triv_{G(\A)}$ can have cohomology in degree
$q=0$, but clearly $\triv$ is not cuspidal. The question whether
there are cohomological cuspidal automorphic representations of
$G(\A)$ for the remaining possible degrees $1\leq q\leq 4$ is more
delicate and we devote the next subsections to it.

\subsection{}
After our consideration of the previous subsection a cuspidal
automorphic representation $\pi$ of $G(\A)$ which has non-trivial
$(\g_s,K)$-cohomology tensorised by $E$ must have a representation
$A_i(\lambda)$, $i=1,2$ as its archimedean component. In case of the
$\Q$-{\it split} general linear group $GL_n$, we know that

\begin{prop}\label{tempgln}
The archimedean component of a cohomological (unitary) cuspidal
representation of $GL_n(\A)$ is tempered.
\end{prop}
See, e.g., \cite{schwSLn}, Thm. 3.3. We will
show in the next theorem that this does not hold for the non-split
inner form $G$. Again, functoriality will be the key-method.

\subsection{Degrees $q=1,4$}\label{sec:cusp}

\begin{thm}\label{thm:nontempcusp}
For all $E$ with highest weight $\lambda=k\omega_2$, $k\geq 0$,
there is a unitary cuspidal automorphic representation $\pi$ of
$G(\A)$ which has cohomology in degrees $q=1,4$. In particular, the
non-tempered representation $A_1(\lambda)$ appears as the
archimedean component of a cohomological cuspidal automorphic
representation and so Proposition \ref{tempgln} cannot be
generalized to inner forms of $GL_n$.
\end{thm}
\begin{proof}
Let $k\geq 0$. By our Proposition \ref{theoremcohunit}, we need to
find a unitary cuspidal automorphic representation $\pi$ of
$GL_2'(\A)$ which satisfies $\pi_\infty|_{G_s(\R)}=J(F_{w_2},1)$ at
the archimedean component. Therefore, let
$\rho_\infty':=\textrm{Sym}^{k+1}\C^2$ be the $k+1$-th symmetric
power of the standard representation of $GL_1'(\R)$. By the local
Jacquet-Langlands correspondence $JL_\infty$ between representations
of $GL_1'(\R)=\H^*$ and $GL_2(\R)$ (cf. \cite{gelbjacquet}, Thm.
8.1) it lifts to a square-integrable, irreducible, unitary
representation $\rho_\infty$ of $GL_2(\R)$. Hence, $\rho_\infty=D_\ell$ for some integer $\ell\geq 2$, $D_\ell$ denoting as usual the discrete series representation of $GL_2(\R)$ of minimal $O(2)$-type $\ell$. \\\\ Now, take a non-archimedean
prime $p_0\in S(D)$ and let $N$ be a positive integer prime to $p_0$. If
$N$ is big enough, which we assume, then there is a non-zero modular
cuspform $f$ of weight $\ell$ and level $N$ of $GL_2(\R)$. We may assume that $f$ is an Eigenfunction for all Hecke operators $T_p$, $(p,N)=1$, cf. \cite{rog} p.21.
As in \cite{rog}, (1.5) $f$ defines a cuspidal automorphic form $\varphi$ for $GL_2(\A)$ and
hence an admissible subrepresentation $\rho$ of $\mathcal
A_{cusp}(GL_2(\Q)\backslash GL_2(\A))$. As proved in \cite{rog}, Prop. 2.13, $\rho$ is in fact irreducible, its archimedean component is $\rho_\infty=D_\ell$ and $\rho_{p_0}$ is a spherical representation of $GL_2(\Q_{p_0})$, whence it is in the
principal series. As a consequence, $\rho_{p_0}$ is not
square-integrable. By the characterization of the image of the
global Jacquet-Langlands correspondence $JL$ from the level of
$GL_1'(\A)$ to the level of $GL_2(\A)$, cf. \cite{gelbjacquet}, Thm.
8.3, $\rho$ can therefore not be of the form $\rho=JL(\rho')$ for
any automorphic representation $\rho'$ of
$GL_1'(\A)$ (although it transfers at the archimedean place!)\\\\
Now take $\sigma$ to be the unique irreducible quotient of

$$\textrm{Ind}^{GL_4(\A)}_{GL_2(\A)\times GL_2(\A)}[|\det|^{\frac{1}{2}}\rho\otimes|\det|^{-\frac{1}{2}}\rho].$$
By \cite{moewalgln} it is a residual automorphic representation of
$GL_4(\A)$. According to \cite{ioan}, Thm. 18.1 there is a unique
square-integrable automorphic representation $\pi$ of $GL_2'(\A)$
which is mapped onto $\sigma$ via the global Jacquet Langlands
Correspondence from the level of $GL_2'(\A)$ to the level of
$GL_4(\A)$, as developed in the aforementioned paper. As $\rho$ is
not in the image of the global Jacquet-Langlands correspondence,
Prop. 18.2 of \cite{ioan} ensures that $\pi$ is cuspidal. But as
$\rho_\infty$ transfers via the local Jacquet-Langlands
correspondence to $\rho_\infty'=\textrm{Sym}^{k+1}\C^2$ we see that
$\pi_\infty|_{G_s(\R)}=J(F_{w_2},1)$, cf. \cite{ioan}, Thm. 13.8.
This proves the theorem.
\end{proof}

\begin{rem}
The number $\ell$ can easily be made explicit and equals $\ell=k+3$. In fact, complexifying the action of $\rho_\infty'$ gives a representation of $SL_1'(\C)=SL_2(\C)$ which restricts to an irreducible representation $\wt\rho_\infty$ of the split real form $SL_2(\R)$. On the other hand, restricting $D_\ell$ to $SL_2(\R)$ defines two irreducible discrete series representations $D^+_\ell$ and $D^-_\ell$ of $SL_2(\R)$. As the local Jacquet-Langlands correspondence at the archimedean place can be characterized as the assigment $JL_\infty$ sending $\rho_\infty'$ to the unique $D_\ell$, which satisfies that $D^\pm_\ell$ and $\wt\rho_\infty$ appear as irreducible subquotients of the same principal series representation of $SL_2(\R)$, we must have $\ell=\dim\wt\rho_\infty+1=k+3$, cf. \cite{knapp}, II \S 5.
\end{rem}

Let us put Theorem \ref{thm:nontempcusp} in a broader context. To that end, recall
the notion of a CAP representation of a general connected reductive
algebraic group $H/\Q$: Therefore, let $H$ be the inner form of a
quasi-split group $\wt H$ and $\wt P$ be a {\it proper} parabolic
subgroup of $\wt H$ with Levi subgroup $\wt L$. A unitary cuspidal
automorphic representation $\pi$ of $H(\A)$ is called CAP, if there
is a unitary cuspidal automorphic representation $\eta$ of
$\widetilde L(\A)$ such that $\pi$ is nearly equivalent (i.e.,
locally equivalent at all but finitely many places) to an
irreducible subquotient of Ind$^{\wt H(\A)}_{\wt P(\A)}[\eta]$.
Philosophically speaking, CAP representations typically look at
almost all places like a residual representation of $\wt H(\A)$,
although they might not be nearly equivalent to a residual
automorphic representation of $H(\A)$.\\\\ There are no CAP
representations of split $GL_n(\A)$ (\cite{moewalgln}), which is
reflected in the Ramanujan Conjecture: It says that a unitary
cuspidal automorphic representation of $GL_n(\A)$ has only tempered
components. (Compare this to Proposition \ref{tempgln} for
cohomological representations). Furthermore, recall that the naively
generalized Ramanujan Conjecture (nGRC) would claim this for groups
different from $GL_n$. This ``conjecture'' nGRC is obviously not
true: $\triv$ is even a CAP-representation of $GL_1'(\A)$, as it is
nearly equivalent to the residual representation $\triv_{GL_2(\A)}$
of $GL_2(\A)$. Although $G=GL_2'$ satisfies Strong Multiplicity One,
our Theorem \ref{thm:q=1} gives a nice non-trivial example of a CAP
representation and at the same time a counterexample to nGRC for
$G$:

\begin{cor}
Let $\pi$ be as constructed in the proof of Theorem
\ref{thm:nontempcusp}. Then $\pi$ is a cohomological
CAP-representation of $G(\A)=GL_2'(\A)$. It is also a counterexample
to the naively generalized Ramanujan Conjecture for the inner form
$G$ of $GL_4$.
\end{cor}
\begin{proof}
This is clear, since $\pi$ is by its very construction nearly
equivalent to the residual automorphic representation $\sigma$ of
$GL_4(\A)$.
\end{proof}

\subsection{Degrees $q=2,3$}

If the highest weight $\lambda$ of $E$ does not satisfy the equation
$\lambda=k\omega_2$, $k\geq 0$, but - in view of Proposition
\ref{theoremcohunit} - we still have $E\cong\check{E}$, we can use
Lefschetz numbers, resp. the twisted Arthur Trace Formula to show
that there is a cuspidal automorphic representation $\pi$ which has
non-vanishing cohomology with respect to $E$, i.e., whose
archimedean component is isomorphic to $A_2(\lambda)$. This approach
goes back to the paper \cite{borlabschw}, resp. work of
D. Barbasch and B. Speh, \cite{barsp}. We allow
ourselves to use freely the standard terminology and notation
concerning the trace formula in this section. For details, see
\cite{borlabschw} or \cite{barsp}.\\\\
Observe that since $E$ is assumed to be self-dual, we can find a
Cartan involution on $G$, agained denoted by $\theta$ which fixes
the highest weight $\lambda$ of $E$, i.e.,
$\lambda\circ\theta=\lambda$. Enlarge $G$ to the algebraic
$\Q$-group $G^+:=G\rtimes\<\theta\>$, where multiplication is
defined as
$(g_1,\theta^i)\cdot(g_2,\theta^j):=(g_1\theta(g_2),\theta^{i+j})$.
It is proved in \cite{barsp}, Thm. IX.4 or Prop. XI.1, that one can
find a certain function $f_\A=\otimes'_p f_p\in C^\infty_c(G^+(\A))$
which satisfies the conditions imposed by the simple trace formula
due to J.-P. Labesse and R. Kottwitz: At the archimedean place and
at two non-archimedean places $f_p$ is a local Lefschetz function,
defined as in \cite{borlabschw}, Prop. 8.4. At the remaining places
$f_p=char(K_p)$ the characteristic function of a certain compact
subgroup $K_p\subseteq G(\Q_p)$. Here, $K_p$ is maximal open compact
for almost all places. But for a finite number of non-archimedean
places $p$ the group $K_p$ has to be chosen with care, in order to
ensure that the geometric side of the trace formula is positive. For
details we refer the reader to \cite{barsp}. \\\\In essence, the
twisted trace formula applied to the special function $f_\A$ says that summing up the values
$$a^G(\gamma)J_G(\gamma,f_\A)=vol(G_\gamma(\Q)\R_+\backslash G_\gamma(\A))\cdot|G_\gamma/Z^\circ_{G}(\gamma)|
\cdot\int_{G^\circ_\gamma(\A)\backslash G(\A)}f_\A(g^{-1}\gamma g)dg,$$
the sum ranging over all conjugacy classes of elliptic elements
$\gamma\in G^+(\Q)$ equals summing up $m_{dis}(\pi)tr(\pi(f_\A))$,
the latter sum ranging over all representations $\pi$ in the
discrete spectrum of $G$. In our case this formula simplifies to

$$\sum_{\gamma\in G^+(\Q)_{ell}}a^G(\gamma)J_G(\gamma,f_\A)=
\sum_{\substack{\pi \textrm{ cuspidal with}\\ \textrm{infinitesimal
char.}=\lambda+\rho}}tr(\pi(f_\A)).$$ This is due to the following
three facts:
\begin{enumerate}
\item $\pi_\infty(f_\infty)=0$ if the infinitesimal character of
$\pi$ does not match the infinitesimal character $\lambda+\rho$ of
$E$ \item As $\dim E >0$, an automorphic representation $\pi$ of
$G(\A)$ appearing in the discrete spectrum with infinitesimal
character $\lambda+\rho$ cannot be a character, as $\pi_\infty$
cannot be one-dimensional. Whence, \cite{barsp} IX.3 Thm. (2) shows
that $\pi$ must already be cuspidal in order to contribute to the
harmonic side of the trace formula.
\item $m_{dis}(\pi)=1$ by the multiplicity one theorem for
discretely occurring automorphic representations of $G(\A)$, see
again \cite{ioan} Thm. 18.1.(b).
\end{enumerate}
Now, the main result of \cite{barsp}, cf. Thm. XI.1., states that
(having chosen $f_\A$ with enough care)
$$\sum_{\gamma\in G^+(\Q)_{ell}}a^G(\gamma)J_G(\gamma,f_\A)>0,$$
so there must exist a cuspidal automorphic representation $\pi$ of
$G(\A)$ which has non-trivial cohomology with respect to $E$. By our
Proposition \ref{theoremcohunit}, it must have archimedean component
$\pi_\infty=A_2(\lambda)$. Summarizing this gives

\begin{thm}\label{thm:q=2}
For all self-dual $E$ with highest weight $\lambda\neq k\omega_2$,
$k\geq 0$, there is a cuspidal automorphic representation $\pi$ of
$G(\A)$ which has cohomology in degrees $q=2,3$. In particular,
$A_2(\lambda)$ appears as the archimedean component of a global
cuspidal representation.
\end{thm}

\subsection{}
Combining Theorems \ref{thm:nontempcusp} and \ref{thm:q=2} we have proved

\begin{thm}
For all finite-dimensional, self-dual, irreducible complex representations $E$ of $G(\R)$
$$H_{cusp}^*(G,E)\neq 0$$
\end{thm}

\section*{Acknowledgements}
I am deeply grateful to Ioan Badulescu and Colette M\oe glin, who
helped me a lot in finding the proof of Theorem
\ref{thm:nontempcusp}.

\end{document}